\documentclass[12pt]{article}
\usepackage{amssymb}

\usepackage{amsthm}
\usepackage{amsmath}

\newtheorem{defi}{\bf Definition}[section]
\newtheorem{theo}[defi]{\bf Theorem}

\newtheorem{lem}[defi]{\bf Lemma}

\title{On the Lie structure of a prime associative superalgebra}

\author{Jes\'us Laliena  }

\date{\quad}

\begin{document}
\maketitle\vspace{-1.5cm}

\begin{abstract}
In this paper some results on the Lie structure of prime superalgebras are discussed. We prove that, with the exception of some special cases,  for a prime superalgebra, $A$, over a ring of scalars $\Phi$ with $1/2\in \Phi$,   if $L$ is a Lie ideal of $A$ and $W$ is a subalgebra of $A$ such that $[W, L]\subseteq W$, then either $L\subseteq Z$ or $W\subseteq Z$. Likewise,  if $V$ is a submodule of $A$ and $[V, L]\subseteq V$, then either $V\subseteq Z$ or $L\subseteq Z$ or there exists an ideal of $A$, $M$, such that $0\not= [M,A]\subseteq V$. This work extends to prime superalgebras some results of I. N. Herstein, C. Lanski and S. Montgomery on prime algebras. 
\end{abstract}

\bigskip

{\parindent= 4em \small  \sl Keywords: associative superalgebras, prime superalgebras, Lie structure}

{\parindent=4em \small \bf Classification MSC 2010: 16W55, 17A70, 17C70}

\bigskip

\section{Introduction.}

\bigskip

An associative superalgebra is just a superalgebra that is associative like an ordinary algebra. If  $A=A_0+ A_1$ is a superalgebra, the elements in $A_0 \cup A_1$ are called homogeneous elements.

\medskip

 It is known that, if we take an associative superalgebra, $A$, and we change the product in $A$ by the superbracket product $[a, b]= ab- (-1)^{\bar a \bar b} ba$, where $\bar a, \bar b$ denotes the degree of $a$ and $b$, homogeneous elements in $A=A_0+A_1$, we obtain a Lie superalgebra, denoted by $A^-$. 
 
 \medskip
 
 The Lie structure of prime associative superalgebras and simple associative superalgebras was investigated by F. Montaner (\cite {M}) and S. Montgomery (\cite {Mo}). On
 superalgebras with superinvolution, several papers have also appeared studying the Lie structure of the skewsymmetric elements in relation to the ideals of the superalgebra (see \cite {Go-S}, \cite {Go-L-S}, \cite {L-S}, \cite {L}, \cite {L-Sa}).
 
 \medskip

In the non-graded case, there is a parallel situation for associative algebras with and without involution and Lie algebras. This fact was first studied by I. N. Herstein (\cite {H1}, \cite {H2}) and W. E. Baxter (\cite {B}), and after by several authors: T. E. Erickson (\cite {Er}), C. Lanski (\cite {La}, W. S. Martindale III and C. R. Miers (\cite {Ma}), \dots

\medskip
 
 The aim of this paper is to prove, in the setting of  prime associative superalgebras over a ring of scalars $\Phi$ with $1/2 \in \Phi$, the following results,  which are well known in the non-graded case. These results were proved by I. N. Herstein for semiprime, 2 torsion free rings (see Theorems 3 and 5, and Lemma 4 in \cite {H3}), and by C. Lanski and S. Montgomery for prime rings without restriction in the characteristic (see Theorems 12 and 13, and Lemma 11 in \cite {La-M}).
 
  \medskip
 
 \begin{lem} 
 Let $R$ be a prime, 2-torsion-free ring and $U$ a Lie ideal of $R$. Suppose that $A$ is an additive subgroup such that $[U, A]\subseteq A$ and $[A, A]\subseteq Z$. Then $A\subseteq Z$.
 \end{lem}
 
 \medskip
 
 \begin{theo} 
 Let $R$ be a prime, 2-torsion-free ring and $W$ a subring of $R$. Suppose that $U$ is a Lie ideal of $R$ such that $[W,U]\subseteq W$. Then either $U\subseteq Z$, or $W\subseteq Z$, or $W$  contains a nonzero ideal of $R$.
 \end{theo}
 
 \medskip
 
 \begin{theo}
 Let $R$ be a prime, 2-torsion-free ring and let $U$ be a Lie ideal of $R$. Suppose that $V$ is an additive subgroup of $R$ such that $[V, U]\subseteq V$. Then either $U\subseteq Z$, or $V\subseteq Z$, or there exists an ideal $M$ of $R$ such that $0\not= [M, R]\subseteq V$.
 \end{theo}
 
 \medskip
 
 These results have been very useful in rings  (see for example   \cite {B-C-F-V}, \cite {B-K-S}, \cite {K-B-M-R}, \dots ), and have also been used  in superalgebras, for example, in the study of the Lie ideals of the set of skewsymmetric elements of an associative superalgebra with superinvolution (see  \cite {Go-L-S}, \cite{L-S}, \cite {L}). As these results have never been proved in superalgebras, we are interested in proving them here. To do that we take advantage of some of the ideas developed in the proofs made in \cite {L}, \cite{H3} and \cite{La-M}. 
 
 \medskip
  
 For a complete introduction to the basic definitions and examples of superalgebras,
superinvolutions and prime and semiprime superalgebras, we refer the reader to \cite {E-L-S},  \cite {Go-S} and \cite{M}.

\medskip

Throughout the paper, unless otherwise stated,  $A$ will denote a nontrivial prime associative
superalgebra  over a commutative unital ring $\phi$ of scalars with ${1\over
2}\in \phi$. By a nontrivial superalgebra we understand a superalgebra with a nonzero odd part. $Z$ will
denote the even part of the center of $A$.

\medskip

If $Z\not= 0$, one can consider the localization $Z^{-1}A=\{z^{-1}a : 0\not= z \in Z, a\in A\}$. If
$A$ is prime, then  $Z^{-1}A$ is a central prime associative superalgebra over the field
$Z^{-1}Z$.  We call this superalgebra  the central closure of $A$. We also say that $A$ is a
central order in $Z^{-1}A$. This terminology is not the standard one, for which the definition involves
 the extended centroid. We say that $A$ is a central order in $C(n)$ if $Z\not= 0$ and $Z^{-1} A$ is isomorphic to the Clifford superalgebra of a non-degenerate quadratic space of dimension $n$ over $Z^{-1}Z$ (see Example 1.5 in \cite {Go-S}).
 
 \medskip

 More precisely, in this paper we mainly prove  three results. Let $A$  be a prime associative superalgebra over a ring of scalars $\Phi$ with $1/2 \in \Phi$, such that $A$ is not a central  order in $C(n)$ $n=1,2,3$,  and let $L$ be  a Lie ideal of $A$ then:
 
 \begin{enumerate}
\item[{\rm (1)}] If $V$is a $\Phi$-submodule of $A$ such that $[V, L]\subseteq V$ and $[V, V]\subseteq Z$, then either $L\subseteq Z$ or $V\subseteq Z$.

\item[{\rm (2)}] If $W$ is a subalgebra of $A$ such that $[W, L]\subseteq W$, then either $L\subseteq Z$, or  $W\subseteq Z$, or $W$ contains a  nonzero ideal of $A$.

\item[{\rm (3)}] If $V$ is a $\Phi$-submodule of $A$ such that $[V, L]\subseteq V$, then either $L\subseteq Z$, or $V\subseteq Z$ or there exists  an ideal $M$ of $A$ such that $0\not= [M, A] \subseteq V$.
\end{enumerate}
 
 In the context of superalgebras when  we say subalgebra, submodule or ideal, we mean graded subalgebra, submodule or ideal, respectively.
 
 \medskip

The following results are instrumental for the paper:

\medskip

\begin{lem} (\cite {H2}, Lemma 1.1.9)
Let $A$ be  a semiprime algebra  and  $L$  a Lie ideal of $A$. If  $[a,[a,L]]=0$, then $[a,L]=0$.
\end{lem}

\medskip

\begin{lem} (\cite {M}, Lemmata 1.2, 1.3)
If $A=A_0 \oplus A_1 $ is a semiprime superalgebra, then  $A_0$ and $A$ are semiprime algebras. Moreover, if $A$ is prime, then 
either $A$ is prime or $A_0$ is prime (as algebras).
\end{lem}

\medskip

\begin{lem}(\cite {Go-L-S}, Theorem 2.1)
Let $A$ be a prime nontrivial associative superalgebra. If $L$ is a Lie ideal of $A$,  then either
$L\subseteq Z$ or $L$ is dense in $A$, except if $A$ is a central order in $C(2)$
\end{lem}

\medskip

\begin{lem}(\cite {L})
Let $A$ be a prime superalgebra,  $L$  a Lie ideal of $A$ such that $L$ is dense in $A$, and $v\in A_i$ such that   $vLv=0$, then $v=0$.
\end{lem}

\medskip

\begin{lem} (\cite {L})
Let $A$ be a prime superalgebra, $L$ a Lie ideal of $A$ such that $L$ is dense in $A$, and $V$ a Lie subalgebra of $A$ such that $[V, L]\subseteq V$. If $v^2=0$ for every $v\in V_i$, then $V_i=0$.
\end{lem}

\medskip

We point out that the bracket product in Lemma 1.1, Theorem 1.2, Theorem 1.3 and Lemma 1.4 is the usual one: $[a,b]=ab-ba$, but the bracket
product in Lemma 1.8 is the superbracket $[x_i,y_j]_s=x_iy_j-(-1)^{ij}y_jx_i$ for $x_i\in
A_i, y_j\in A_j$ homogenous elements. In fact, the superbracket product coincides with the usual
bracket if one of the arguments belongs to the even part of $A$. In the following, to simplify the
notation, we will denote both in the usual way $[\ , \ ]$ but we will understand that it is the
superbracket if we are in a superalgebra.

\medskip

Also, from now on, by an element $a\in M$, with $M$ any  $\Phi$-submodule of a superalgebra $A$, we will always understand  a homogenous element $a\in M$, that is, $a\in M_0 \cup M _1$, unless otherwise stated. 

\bigskip

\section {Lie structure of an associative superalgebra.}

\bigskip
 
Let $A$ be an associative superalgebra and $M$ be a $\Phi$-submodule of $A$. Denote by $\overline M$ the subalgebra of $A$ generated by $M$. We will say that $M$ is dense in $A$ if $\overline M$ contains a nonzero ideal of $A$.

\begin{lem}
Let $A$ be a prime superalgebra such that it is not an order in $C(2)$. Let $L$ be a Lie ideal of $A$. Then either $L\subseteq Z$ or $C(L) \subseteq Z$, where $C(L)= \{ x \in A : [x, L]=0\}$.
\end{lem}

\begin{proof}[Proof:]
We notice that $C(L)$ is a Lie ideal and a subalgebra of $A$.  Indeed, let $x, y \in C(L)$, $a\in A$ and $u\in L$,
$$ [[x,a],u]= - (-1)^{\bar x \bar a+ \bar x \bar u}[[a,u],x ] - (-1)^{\bar u \bar x + \bar u \bar a } [[u, x], a]=0$$
$$[xy, u]= x[y,u] + (-1)^{\bar u \bar y}[x, u]y=0.$$
{\noindent So, by Theorem 4.1 and its proof  in \cite {Go-S} either  $C(L)\subseteq Z$ or $C(L)$ is dense in $A$. But if $C(L)$ is dense in $A$, then there exists a nonzero ideal $I$ of $A$ such that $[I,L]=0$, and from Lemma 2.3 in \cite {L-S} $L\subseteq Z$}

\end{proof}

\begin{lem}
Let $A$ be a prime superalgebra such that it is not an order in $C(n), n= 1, 2, 3$. Let $W$ be a subalgebra of $A$ and a Lie ideal of $[A,A]$. Then either $W\subseteq Z$ or $W$ is dense in $A$.
\end{lem}

\begin{proof}[Proof:]
From Theorem 3.3 in \cite {M}, we know that either $W\subseteq Z$ or there exists an ideal $I$ of $A$ such that $0\not=[I, A]\subseteq W$. Suppose that $W\nsubseteq Z$ and so $I$ is an ideal of $A$ such that $0\not= [I, A]\subseteq W$.  Notice that $[I,A]$ is a nonzero Lie ideal of $A$. Therefore, by Lemma 1.6, either $[I,A]\subseteq Z$ or $[I,A]$ is dense in $A$. If $[I,A]$ is dense in $A$, then $W$ is dense in $A$. If $[I, A]\subseteq Z$, we can localize $A$ by $Z$ and consider $Z^{-1}A$. Then $0\not=[Z^{-1}I, Z^{-1}A]\subseteq Z^{-1}Z$. Therefore $Z^{-1}I$ has invertible elements and so $Z^{-1}I = Z^{-1}A$. But then, since $[I,A]\subseteq Z$, $[Z^{-1}A, Z^{-1}A]\subseteq Z^{-1}Z$, that is, $[ [Z^{-1}A, Z^{-1}A], Z^{-1}A]=0$. From Lemma 2.6 in \cite {M}, $A$ is $C(n)$ with $n= 1, 2 $ or $3$, a contradiction.
\end{proof}

\begin{lem}
Let $A$ be a prime superalgebra such that it is not an order in $C(n)$ with $n=1, 2,3$, and $L, U$  Lie ideals of $A$ such that $[L, U]\subseteq Z$. Then either $L\subseteq Z$ or $U\subseteq Z$. 
\end{lem}

\begin{proof}[Proof:]
Suppose that $ L \nsubseteq Z$. Since $[L, U]\subseteq Z$, it follows that  $[L_0,U_1]=[L_1, U_0]=0$ and $[L_0,U_0]+ [L_1, U_1]\subseteq Z$. So for every $u\in U_0$ we have $[u, [u, L]]=0$, and from Lemmata 1.4 and 1.5 we deduce that $[U_0, L_0]=0$. But $[U_0, L_1]=0$ and so $[U_0, L]=0$. From Lemma 2.1 $U_0\subseteq Z$, and $[L,U]=[L_1, U_1]\subseteq Z$. If $[L, U]=0$, then by Lemma 2.1  $U\subseteq Z$. And if $0\not=[L, U]\subseteq Z$, then $Z\not=0$ and we can consider the localization $Z^{-1}A$ and the Lie ideals $Z^{-1} ZL, Z^{-1} ZU$ in $Z^{-1}A$.

We suppose now that $L\nsubseteq Z$ and $U\nsubseteq Z$. From Theorem 3.2 in \cite {M} there exist  nonzero ideals $I, J$ of $A$ such that 
$$0\not= [I, A]\subseteq L, \quad 0\not= [J, A]\subseteq U.$$
{\noindent Notice that if $[I,A]=0$ or $[J, A]=0$, then, by Lemma 2.3 in \cite  {L-S},  $A\subseteq Z$, a contradiction. Since $[L, U]\subseteq Z$ we have}
$$[[Z^{-1}ZL, Z^{-1}ZU]\subseteq Z^{-1}Z,$$
{\noindent  and so }
$$[[Z^{-1} I, Z^{-1} A], [Z^{-1} J, Z^{-1} A]]\subseteq Z^{-1}Z.$$

  If $[[Z^{-1} I, Z^{-1} A], [Z^{-1} J, Z^{-1} A]]\not= 0$ then $Z^{-1}I, Z^{-1}J$ have invertible elements and $Z^{-1} I, Z^{-1}J = Z^{-1}A$. Therefore 
  $$ [[[Z^{-1} A, Z^{-1} A], [Z^{-1} A, Z^{-1} A]], Z^{-1}A]= 0.$$ 
  {\noindent Now, from Lemma 2.6 in \cite {M} we have a contradiction with our hypothesis about $A$ not being a central order in $C(n)$ with $n = 1, 2, 3$ (notice that in \cite {M} the product $a\circ b$ is our product $[a, b]$ when $a, b\in A_1$).  }
  
  And if $[[Z^{-1} I, Z^{-1} A], [Z^{-1} J, Z^{-1} A]]=0$, then, by Lemma 2.1, 
  $$ \text {either}\  [Z^{-1} I, Z^{-1} A]  \subseteq Z^{-1}Z \ \text { or} \ [ \ Z^{-1} J, Z^{-1} A]\subseteq Z^{-1}Z.$$ {\noindent Therefore, since $[I, A]\not= 0, [J, A]\not=0$, }
   $$ \text {either} \ Z^{-1}I = Z^{-1}A \ \text{ or} \ Z^{-1}J=Z^{-1}A.$$ 
   {\noindent Since $[Z^{-1} I, Z^{-1}A], [Z^{-1}J, Z^{-1}A]\subseteq Z^{-1}Z$,  in both cases we have}
    $$[[Z^{-1} A, Z^{-1} A],  Z^{-1} A]=0.$$
    {\noindent Again from Lemma 2.6 in \cite {M} we have a contradiction with our hypothesis.}
\end{proof}

\begin{lem}
Let $A$ be a prime superalgebra such that it is not an order in $C(n)$ with $n=1, 2,3$, and $L$  a Lie ideal of $A$ such that $[t, L]\subseteq Z$ with $t\in A$. Then either $t\in Z$ or $ L\subseteq Z$. 
\end{lem}

\begin{proof}[Proof:]
Consider $U=\{ x \in A : [x, L]\subseteq Z\}$. We notice that $U$ is a $\Phi$ - submodule of $A$, and it is also  a Lie ideal because for every $u\in L$, $x\in U$ and $y\in A$
$$ [[x, y],u]= (-1)^{\bar u \bar y} [[x,u],y] + (-1)^{\bar y \bar u} [x, [y,u]]\in Z.$$
{\noindent So, $U$ is a Lie ideal of $A$ and from Lemma 2.3  either $U\subseteq Z$ or $L\subseteq Z$.}
\end{proof}

\begin{lem}
Let $A$ be a prime superalgebra such that it is not an order in $C(n)$ with $n=1, 2,3$,  $L$  a Lie ideal and $V$ a $\Phi$ -submodule of $A$ such that $[V,L]\subseteq V $ and $[V,V]\subseteq Z$. Then either $L\subseteq Z$ or $V\subseteq Z$.
\end{lem}

\begin{proof}[Proof:]
Suppose that $L\nsubseteq Z$. Then, from Theorem 3.2 in \cite {M}, there exists a nonzero ideal $I$ of $A$ such that $ [I,A]\subseteq L$, and $[I, A]\not= 0$ by Lemma 2.3 in \cite {L-S}. 

 If $I\cap Z\not=0$, we localize $A$ by $Z$ and then $Z^{-1} Z \cap Z^{-1}I \not=0$, so $Z^{-1} I$ has invertible elements and $Z^{-1}I= Z^{-1}A$. Hence $Z^{-1}ZV$ is a Lie ideal of $[Z^{-1}A, Z^{-1}A] $.  From Theorem 3.3 in \cite {M} either $Z^{-1}ZV\subseteq Z^{-1}Z$ or there exists a nonzero ideal $N$ of $A$ such that $0\not=[Z^{-1}N, Z^{-1}A]\subseteq Z^{-1}ZV$. In the second case, since $[V,V]\subseteq Z$, we have 
 $$[[Z^{-1}N, Z^{-1}A], [Z^{-1}N, Z^{-1}A]]\subseteq Z^{-1} Z.$$ 
 {\noindent From Lemma 2.4 we have $[Z^{-1}N, Z^{-1}A] \subseteq Z^{-1}Z$, and since $[Z^{-1}N, Z^{-1}A]\not= 0, Z^{-1}N=Z^{-1}A$. So, }  
 $$ [Z^{-1}A, Z^{-1}A]\subseteq Z^{-1}Z, $$
 {\noindent and by Lemma 2.3, $Z^{-1}A\subseteq Z^{-1}Z$, a contradiction with our assumptions. Therefore $Z^{-1} ZV \subseteq Z^{-1}Z$ and so $V\subseteq Z$.
}

 If $I\cap Z= 0$, then for every  $v\in V_0$ we have  
 $$[v,[v,[I, A]_0]]\subseteq [V,V]\cap I \subseteq Z \cap I =0.$$ 
 {\noindent From Lemmata 1.4 and 1.5 we have}
  $$[V_0, [I, A]_0]=0.$$
  {\noindent  Now, we consider $W= [V, [I,A]]$. Notice that}
 $$ [W,W]\subseteq [V, V]\cap I \subseteq Z\cap I=0.$$
 {\noindent So for every $w\in W_1 $ we have $w^2=0$.  From Lemma 1.8 $W_1=0$. Therefore} 
 $$W_1=[V_0, [I, A]_1]+ [V_1, [I, A]_0]=0 .$$
 {\noindent We have $0\not= [I, A]$, and also $[I, A]\nsubseteq Z$, because if $[I, A]\subseteq Z, $ then $[I, A]\subseteq Z\cap I=0$, a contradiction.  Therefore, since $[V_0, [I, A]]=0$, we have  $V_0\subseteq Z$ because of Lemma 2.1. But we have deduced that $[V_1, [I, A]_0]=0$, and we observe that}
  $$[V_1, [I, A]_1] \subseteq V_0 \cap I \subseteq Z\cap I=0.$$
  {\noindent Therefore, we also obtain  that $[V_1, [I,A]]=0$ and, again by Lemma 2.1, $V_1\subseteq Z$,  that is, $V\subseteq Z$.
 }
 
\end{proof}

\bigskip

We prove now our first theoremt.

\bigskip

\begin{theo}
Let $A$ be a prime superalgebra such that it is not an order in $C(n)$ for $n=1, 2, 3$. Let $W$ be  a subalgebra of $A$, $L$ a Lie ideal of $A$ and $[W,L]\subseteq W$. Then either $L\subseteq Z$, $W\subseteq Z$ or $W$ is dense in $A$.
\end{theo}

\begin{proof}[Proof:]
We suppose that $L\nsubseteq Z$. Because of Lemma 1.6 there exists a nonzero ideal $N$ of $A$ such that $N\subseteq \bar L$. 

\medskip

Let $V= [W, L]$. If $V=0$, then from Lemma 2.1 we have either $W\subseteq Z$ or $L\subseteq Z$. Since $L\nsubseteq Z$ we deduce that $W\subseteq Z$.

\medskip

So, suppose now that $V\not= 0$.  Let $0\not= u\in V, w\in W$. We notice that if $u$ is even,  then $[u, u]=0$, and if $u$ is odd,  then $[u, u]=0$ implies that $u^2=0$.  We will prove that  if $t, s \in W$ and $u\in V,$ with $[u, u]=0$,  such that  $[t,u][u,s]\not=0$, then $W$ is dense. First we see that  $[t,u][u,s]A\subseteq W$. We have 
$$[u,s]a= [u,sa]-(-1)^{\bar s \bar u}s[u,a]$$
{\noindent  for every  $a\in A$, therefore}
$$[t,u][u,s]a=[t,u][u,sa]-(-1)^{\bar u \bar s}[t,u]s[u,a].$$
{\noindent  But}

\begin{eqnarray*}
[t,u][u,sa]&=&[t,u[u,sa]]-(-1)^{ \bar t \bar u}u[t,[u,sa]]\\&=&(-1)^{\bar u}[t,[u,usa]]-(-1)^{\bar t \bar u} u[t,[u,sa]]\in W,
\end{eqnarray*}
{\noindent because $W $ is a subring and $L$ is a Lie ideal of $A$. And also }
$$[t,u]s[u,a]=[t,u][s,[u,a]]+(-1)^{\bar s (\bar u+ \bar a)}[t,u][u,a]s \in W , $$
{\noindent because }
\begin{eqnarray*}
[t,u][u,a]&=&[t,u[u,a]]-(-1)^{\bar t \bar u}u[t,[u,a]]\\&=&(-1)^{\bar u} [t,[u,ua]]-(-1)^ {\bar t \bar u}u[t, [u,a]] \in W.
\end{eqnarray*}
{\noindent Therefore $[t,u]s[u,a]\in W$, and so}
 $$[t,u][u,s]a \in W \ \text{ for every } \   t, s \in W, a\in A, u\in V, \ \text {with } \ [u, u]=0 .$$

{\noindent Next we will show that }
$$\overline{L} [t,u][u,s]A\subseteq W.$$
{\noindent Since $[W, L] \subseteq W$ it follows that }
$$L[t,u][u,s]A\subseteq [ L ,[t,u][u,s]A]+[t,u][u,s]A L\subseteq W.$$ 
{\noindent Notice also that }
$$L^2[t,u][u,s]A\subseteq [L,W]+[t,u],[u,s]A \subseteq W.$$
{\noindent  Using induction over $i$ it is easy to prove that }
$$L^i[t,u][u,s]A\subseteq W.$$ 
{\noindent and so that }
$$\bar L[t,u][u,s]A\subseteq W.$$
{\noindent Hence, since $N$ is a nonzero ideal such that $N\subseteq \bar L$, we have $M=N[t,u][u,x]A$, a nonzero ideal such that $M\subseteq W$.}

\medskip

 Therefore, either $W$ is dense in $A$, or, if $W$ is not dense in $A$,  
$$[t,u][u,s]=0 \ \text {for every} \ t,s \in W, u\in V  \ \text {such that} \  [u, u]=0,$$
{\noindent because of the primeness of $A$.}

\medskip

 We suppose now that $W$ is not dense, and so $[t,u][u,s]=0$ for every $u\in V$ such that $[u, u]=0$, and for every $ t,s \in W$. We will show that $V=[W, L]=0$, a contradiction with our assumption. We prove this in 4 steps. Let $K=[V, V]$.
\medskip

1. $K=[V, V]= [V_1, V_1]$. Indeed, let $x, y, u \in V$ such that  $u^2=0$. From our assumption $[u,x][u, y]=0$, and expanding this gives 
$$ux u y- (-1)^{\bar y \bar u}u x yu+(-1)^{\bar x \bar u + \bar y \bar u}xuyu=0.$$
{\noindent  Right multiplication by $u$ gives $uxu yu=0$. Since $[y, l]\in V$ for every $l\in L$, we obtain that $[y, [l,u]]\in V$. So $ux u [y, [l,u]]u=0$. Expanding this expression yields $ux uluyu=0$.  From Lemma 1.7 we deduce that $uVu=0$. If $u, u^\prime \in V$ are homogeneous elements  with $u^2=(u^\prime)^2=0$, we conclude that}
$$(uu^\prime)^2 = uu^\prime uu^\prime \in uVuu^\prime=0.$$
{\noindent If $l\in L$ we have }
$$ 0=u[u^\prime, l]uu^\prime= uu^\prime l u u^\prime . $$
{\noindent So $uu^\prime L uu^\prime=0$,  and,  from Lemma 1.7, }
$$uu^\prime =0 \quad \text {for every}  \quad u, u^\prime \in V, \text {   homogeneous,  with} \quad u^2=(u^\prime )^2 =0. \quad (*)$$
{\noindent  Now consider $x, y \in V_1, u, v\in V_0$. We have $[x,u]^2=0=[y,v]^2$, and so $[x,u][y,v]=0$, because of $(*)$. Since $[V_0, V_1]$ is additively generated by the elements $[x,u]$ with $x\in V_1, u \in V_0$, we have $v^2=0$ for every $v\in[V_0, V_1]$. From Lemma 1.8, }
$$[V_0, V_1]=0,$$
{\noindent  and }
$$[V, V]=[V_0, V_0] + [V_1, V_1].$$
{\noindent Now consider $X=[V_0, V_0]$. We notice that $X$ is a Lie subalgebra of $A$ and $[X, L ]\subseteq X$. From our assumption for every $ x, y, u, v \in V_0$ we have $[x,u]^2=[y,v]^2=0$, and so, by $(*)$ we obtain that $[x,u] [y,v]=0$. Again, since $[V_0, V_0]$ is additively generated by the elements $[x,u]$ with $x, u \in V_0$, we  deduce that for every $v\in X$, $v^2=0$. From Lemma 1.8, }
$$X=[V_0, V_0]=0.$$
{\noindent Therefore $[V, V]= [V_1, V_1]$.}

  \medskip
  
  2. $K=[V_1, V_1]\subseteq Z$. From Lemma 1.5, $A_0$ is semiprime.  Also, we notice that $L_0$ is a Lie ideal of $A_0$, and it is satisfied that 
  $$K=[V,V]=[V_1, V_1]\subseteq L_0, \quad [K, L_0]\subseteq K  \quad  \text {and} \quad  [K, K] \subseteq [V_0, V_0]=0.$$
 {\noindent  From Lemma 4 in \cite {H3},  $[K,L_0]=0$. Moreover, since}
  $$[K, L_1]\subseteq [[V_1, V_1], L_1] \subseteq [V_1, V_0]=0,$$
  {\noindent we deduce that $[K, L]=0$. From Lemma 2.4, $K\subseteq Z$.}
  
  \medskip
  
 3. $K=[V,V]=0$. Indeed, if $K\not= 0$,  then $Z\not= 0$, and we can localize $A$ by $Z$ and consider $Z^{-1}A, Z^{-1}ZW $ and $Z^{-1}ZL$. From Theorem 3.2 in \cite {M}, there exists an ideal of $A$, $I$, such that $0\not=[Z^{-1} I, Z^{-1}A] \subseteq Z^{-1}ZL$. Notice that $[Z^{-1}I, Z^{-1}A]$ is a Lie ideal of $Z^{-1}A$. We claim that $Z^{-1}I = Z^{-1} A$. 
 
 To prove this, we distinguish two cases: when $[Z^{-1}I, Z^{-1}A]\subseteq Z^{-1}Z$, and when $[Z^{-1}I, Z^{-1}A]\nsubseteq Z^{-1}Z$.
  
  If $[Z^{-1}I, Z^{-1}A]\subseteq Z^{-1}Z$, $Z^{-1}I$ has invertible elements and then $Z^{-1}I= Z^{-1}A$.   
  
  If $[Z^{-1}I, Z^{-1}A]\nsubseteq Z^{-1}Z$, then, since $K\subseteq Z$, we have 
  $$[[[Z^{-1}I, Z^{-1}A], Z^{-1}ZW], [[Z^{-1}I, Z^{-1}A], Z^{-1}ZW]]\subseteq Z^{-1}Z.$$
  {\noindent Notice that if $[[[Z^{-1}I, Z^{-1}A], Z^{-1}ZW], [[Z^{-1}I, Z^{-1}A], Z^{-1}ZW]]=0$, using Lemma 2.5 for $L=[Z^{-1}I, Z^{-1}A]$ and $V=[[Z^{-1}I, Z^{-1}A], Z^{-1}ZW]$, we obtain that}
  $$[[Z^{-1}I, Z^{-1}A], Z^{-1}ZW]]\subseteq Z^{-1}Z.$$
  {\noindent If $[[Z^{-1}I, Z^{-1} A], Z^{-1}ZW]\not= 0$, then, since $[[Z^{-1}I, Z^{-1} A], Z^{-1}ZW]\subseteq Z^{-1}I$ we have $Z^{-1}I=Z^{-1}A$. And if $[[Z^{-1}I, Z^{-1}A], Z^{-1}ZW]=0$, then, by Lemma 2.1, $[Z^{-1}I, Z^{-1} A]\subseteq Z^{-1}Z$ or $Z^{-1}ZW\subseteq Z^{-1}Z$. Since $V\not= 0$, $Z^{-1}ZW\nsubseteq Z^{-1}Z$, and so}
   $$0\not=[Z^{-1}I, Z^{-1} A] \subseteq Z^{-1}Z \cap Z^{-1}I,$$
   {\noindent that is, $Z^{-1}A= Z^{-1}I$. So, finally,}
   
   $$0\not= [[[Z^{-1}I, Z^{-1}A], Z^{-1}ZW], [[Z^{-1}I, Z^{-1}A], Z^{-1}ZW]]\subseteq Z^{-1}Z.$$
   {\noindent Then $Z^{-1}I$ has invertible elements and $Z^{-1}I= Z^{-1}A$ }
  
   So  $Z^{-1}I = Z^{-1} A$, and then $[Z^{-1}A, Z^{-1}A]\subseteq Z^{-1}ZL$. Therefore  $Z^{-1}ZW$ is a subalgebra and a Lie ideal of $[Z^{-1}A, Z^{-1}A]$. From Lemma 2.2, either $Z^{-1}ZW \subseteq Z^{-1}Z$ or $Z^{-1}ZW$ is dense in $Z^{-1}A$. If $Z^{-1}ZW\subseteq Z^{-1}Z$, then $W\subseteq Z$, a contradiction because then $V=0$. Therefore $Z^{-1}ZW$ is dense in $Z^{-1}A$, and there exists a nonzero ideal $J$ of $A$ such that $Z^{-1}J\subseteq Z^{-1}ZW$. Hence, since $K\subseteq Z$, 
   $$[[Z^{-1}J, Z^{-1}ZL], [Z^{-1}J, Z^{-1}ZL]]\subseteq Z^{-1}Z.$$
  
  We observe that if  $[[Z^{-1}J, Z^{-1}ZL], $ $[Z^{-1}J, Z^{-1}ZL]]=0$, then by Lemma 2.1
    $$[Z^{-1}J, Z^{-1}ZL]\subseteq Z^{-1}Z.$$
   {\noindent From Lemma 2.3, either $Z^{-1}J \subseteq Z^{-1}Z$ or $Z^{-1}ZL \subseteq Z^{-1}Z$. In the first case, we obtain that $(Z^{-1})(_Z^{-1}(A_1 + A_1^2)=0$, a contradiction with the primeness. In the second, $L\subseteq Z$ and $V=0$, again a contradiction.}
   
    And if $[[Z^{-1}J, Z^{-1}ZL], [Z^{-1}J, Z^{-1}ZL]]\not=0$, since $[[Z^{-1}J, Z^{-1}ZL], [Z^{-1}J, Z^{-1}ZL]]$ $\subseteq Z^{-1}Z$, then $Z^{-1}J$ has invertible elements, and $Z^{-1}J=Z^{-1}A$. That is, 
    $$Z^{-1}ZW=Z^{-1}A.$$ 
    {\noindent But  then }
    $$[Z^{-1}A, Z^{-1}A]= [Z^{-1}I, Z^{-1}A]\subseteq Z^{-1}ZL,$$
    
   {\noindent and} 
   $$K=[V, V]=[[W, L], [W, L]]\subseteq Z,$$ 
   {\noindent implies that}
   $$ [[Z^{-1}A, [Z^{-1}A, Z^{-1}A]], [Z^{-1}A, [Z^{-1}A, Z^{-1}A]]]\subseteq Z^{-1}Z.$$
   {\noindent From Lemma 2.3, }
   $$[Z^{-1}A, [Z^{-1}A, Z^{-1}A]]\subseteq Z^{-1}Z,$$
   {\noindent and again by Lemma 2.3, $Z^{-1}A\subseteq Z^{-1}Z$, a contradiction. So $K=[V, V]=0$.}
      
  \medskip
  
  4. Finally, we reach a contradiction.  $V$ is $\Phi$-submodule of $A$ and $[V,L]\subseteq V$ and $[V,V]=0$ by step 3.  From Lemma 2.5 we have $V=[W,L]\subseteq Z$, because $L\nsubseteq Z$. Then by Lemma 2.4  $W\subseteq Z$, a contradiction because $V\not=0$. 
 
\end{proof}

\bigskip

And, now, to finish, we prove our second theorem.

\bigskip

\begin{theo}
Let $A$ be a prime superalgebra such that it is not an order in $C(n)$ with $n= 1, 2, 3$.  Let $L$ be a Lie ideal of $A$ and $V$ a $\Phi$-submodule of   $A$ such that $[V, L]\subseteq V$. Then either $L\subseteq Z$ or $V\subseteq Z$ or there exists an ideal $M$ of $A$ such that $[M, A]\subseteq V$.
\end{theo}

\begin{proof}[Proof:]
Let $K=[V, L]$, and $T= \{ x\in A: [x, A]\subseteq V\}$. Then $T$ is a subalgebra of $A$ because for every $t, s \in T$ and $a\in A$
$$ [ts,a] = [t, sa] + (-1)^{\bar t \bar s + \bar a \bar t} [s, at] \in V.$$

Since
$$[[K,K], A]\subseteq [[K,A],K]\subseteq [L, V]\subseteq V,$$
{\noindent it follows that $[K, K]\subseteq T$. If we consider $T^\prime$, the subring generated by  $[K,K]$, we have $[T^{\prime}, L]\subseteq T^\prime$, because}
\begin{eqnarray*}
[[[K,K],L],A]&\subseteq & [[K,K],[L,A]] + [[[K,K],A,],L]\\ & \subseteq & [[K,K],L]+[V,L]\subseteq  V,
\end{eqnarray*}
{\noindent and because for every  $t, s \in [K,K]$ and $u\in L$ we have}
$$ [ts, u]= t[s,u] + (-1)^{\bar s \bar u} [t,u]s \in T^{\prime}.$$
{\noindent Now, $T^\prime$ is a subalgebra of $A$ and $[T^\prime, L]\subseteq T^\prime$. From Theorem 2.6 either $L\subseteq Z$, or $T^\prime \subseteq Z$ or $T^\prime$ contains a nonzero ideal $M$ of $A$. If $L\subseteq Z$ we have finished. If $L\nsubseteq Z$ and $T^\prime \subseteq Z$, then $[K,K]\subseteq Z$ and then  $K\subseteq Z$ by Lemma 2.5. So $[V, L]\subseteq Z$, but then by Lemma 2.4 $V\subseteq Z$. If $M$ is an ideal of $A$ such that $M\subseteq T^\prime$, then $M\subseteq T$ and $[M, A]\subseteq V$.}

\end{proof}

\bigskip

\author{Jes\'us Laliena  \footnote {The author has been
supported by the Spanish Ministerio de Ciencia e Innovaci\'on (MTM 2010-18370-CO4-03).} 
\\{\small Departamento de Matem\'aticas y Computaci\'on}\\
{\small  Universidad de La Rioja}\\
{\small  26004, Logro\~no. Spain}\\
{\small jesus.laliena@dmc.unirioja.es }}

\end{document}